\documentclass[10pt,oneside,reqno]{amsproc}
\usepackage{amsmath,amssymb,amsthm,mathrsfs}
\usepackage{microtype}
\usepackage{xcolor}
\usepackage{hyperref}
\hypersetup{pdftitle={Prevalence of unmarked length spectral rigidity for negatively curved metrics},colorlinks=true,linkcolor=red,citecolor=green!40!black,urlcolor=blue!60!black}
\numberwithin{equation}{section}

\newtheorem*{maintheorem}{Main Theorem}
\newtheorem{theorem}{Theorem}[section]
\newtheorem{proposition}[theorem]{Proposition}
\newtheorem{lemma}[theorem]{Lemma}
\newtheorem{corollary}[theorem]{Corollary}
\theoremstyle{definition}

\theoremstyle{remark}
\newtheorem{remark}[theorem]{Remark}
\newcommand{\norm}[1]{\left\|#1\right\|}
\newcommand{\eps}{\varepsilon}
\DeclareMathOperator{\dist}{dist}

\newcommand{\hide}[1]{}

\DeclareMathOperator{\cU}{{\mathscr U}}

\renewcommand{\ge}{\geqslant}
\renewcommand{\le}{\leqslant}

\allowdisplaybreaks[2]
\date{}

\title[Prevalence of sparsity for negatively curved metrics]{Prevalence of sparsity for negatively curved metrics}

\author[Kostiantyn Drach]{Kostiantyn Drach}
\author[Vadim Kaloshin]{Vadim Kaloshin}

\thanks{\ The authors are supported by ERC Advanced Grant ``SPERIG'' (\#885707). The first author is also partially supported from grants 2021 SGR 00697 (Generalitat de Catalunya), CEX2020-001084-M (Maria de Maeztu Excellence program), PID2023-147252NB-I00 and CNS2025-166633 (AEI).}

\address{Universitat de Barcelona, Gran Via de les Corts Catalanes, 585, 08007 Barcelona, Spain}
\address{Centre de Recerca Matem\`atica, Edifici C, Carrer de l'Albareda, 08193 Bellaterra, Spain}
\email{kostiantyn.drach@ub.edu}
\address{Institute of Science and Technology Austria, Am Campus 1, 3400 Klosterneuburg, Austria}
\email{vadim.kaloshin@gmail.com}

\begin{document}

\begin{abstract}
Let $M$ be a closed manifold of arbitrary dimension. The \emph{length spectrum} of a negatively curved metric $g$ on $M$ is the set of lengths of all possible closed geodesics of $g$. We prove that metrics whose length spectrum is \emph{exponentially sparse}, i.e., the gaps between distinct lengths are bounded below by a quantity decaying exponentially in the length, are \emph{prevalent}, and hence \emph{dense}, among $C^k$-smooth negatively curved metrics on $M$ for sufficiently large $k$. Moreover, unlike in all known constructions, the exponent in our bound grows only sublinearly in the regularity $k$. This work is motivated by our recent result on prevalent unmarked length spectral rigidity for expanding circle maps and a very recent result by DeWitt, Durham, Reber, and O'Hare that $C^k$-smooth negatively curved metrics with exponentially sparse spectra are locally rigid.
\end{abstract}

\maketitle

\section{Introduction}
Let $M$ be a closed connected smooth manifold of dimension at least two, and let $g$ be a smooth metric on $M$ with negative sectional curvature. Every nontrivial free homotopy class $\gamma$ of curves on $M$ contains a unique closed geodesic of $g$. We write $\ell_g(\gamma)$ for the length of such geodesic. The set of all possible lengths of closed geodesics, with multiplicities, is called the \emph{length spectrum} of $g$. One of the open questions is whether the length spectrum of $g$ defines $g$ up to smooth changes of coordinates. Motivated by this question, as well as recent results on a related question on prevalent local rigidity for expanding circle maps \cite{DK}, in this note we study the structure of the length spectrum.

To state our main result, we need to recall the following definition. Denote by $C^k(S^2T^*M)$ the space of all covariant $C^k$-smooth symmetric tensors on $M$ equipped with the $C^k$ norm. Following Hunt--Sauer--Yorke~\cite{HSY}, a Borel subset $E$ of $C^k(S^2T^*M)$ is \emph{shy} if there is a compactly supported Borel probability measure $\mu$ on this Banach space such that $\mu(a+E)=0$ for every $a\in C^k(S^2T^*M)$. This probability measure $\mu$ is the product of Lebesgue probability measures on a properly chosen Hilbert cube of local Fourier harmonics. A Borel set is \emph{prevalent in an open set} $\cU \subset C^k(S^2T^*M)$ if its complement in $\cU$ is shy in the ambient Banach space. It follows that a prevalent set in $\cU$ is also dense there.

Our main result is the following quantitative prevalence result, localized around an arbitrary smooth metric $h$ of negative sectional curvature; we impose no conditions on this metric beyond regularity. To the best of our knowledge, this is the first construction where the exponent in the sparsity condition grows sublinearly with regularity. 

\bigskip
\bigskip

\begin{maintheorem}[Prevalence of exponential gaps]
\label{thm:prevalence}
For every metric $h \in C^\infty(S^2T^*M)$ of negative sectional curvature, there exist an open $C^3$-neighborhood $\cU$ of $h$ and a constant $B=B(h)>0$ such that, for every integer $k\ge3$, the set of metrics $g\in\cU\cap C^k(S^2T^*M)$ for which there is $a = a(g)>0$ satisfying the \emph{sublinear sparsity condition}
\[
 |\ell_g(\gamma)-\ell_g(\eta)|
 \ge a \cdot e^{-B\sqrt{k}\max\{\ell_g(\gamma),\ell_g(\eta)\}},
 \qquad \forall \gamma\ne\eta
\]
is Borel, prevalent, and hence dense in $\cU\cap C^k(S^2T^*M)$. 
\end{maintheorem}

Together with a very recent result of \cite{DDRO}, where the authors prove that $C^k$-smooth negatively curved metrics whose length spectra satisfy an exponential sparsity condition are locally rigid, this Main Theorem may imply that local length spectral rigidity of $C^k$-smooth negatively curved metrics is prevalent.

In \cite{KH}, it was shown that a prevalent diffeomorphism of a unit ball in $\mathbb R^n$ has at most stretched exponential growth of the number of periodic points. The notion of prevalence, as well as Borel--Cantelli-type arguments for estimating the probability of elementary events, is similar to the present paper.

\subsection*{Previous results}

Dolgopyat--Jakobson~\cite{DJ} proved exponential lower bounds between unequal lengths for hyperbolic manifolds whose holonomy groups admit algebraic generators. This includes every closed hyperbolic manifold of dimension at least three and a dense family of hyperbolic surfaces in Teichm\"uller space. In variable negative curvature, their Theorem~4.1 shows that gaps smaller than any prescribed rapidly
decreasing function occur along an unbounded sequence for a Baire-generic
set in the $C^k$ topology, $k>3$. Schenck~\cite[Theorem~1]{Schenck} proves that metrics with exponentially separated distinct length values are dense in the $C^k$ topology, $k\ge2$. His construction uses successive local conformal perturbations and gives a separation exponent growing linearly with $k$. 

Our probabilistic construction in this note mimics the analogous prevalence result about the length (log-multiplier) spectrum for expanding circle maps~\cite[Appendix A]{DK} (see the full comparison at the beginning of Section~\ref{prob:prevalence-section}).

\subsection*{Strategy of the proof}

Fix two closed geodesics of lengths at most $T$. 
\begin{enumerate}
\item 
\emph{Find a region with few nearby passages.} 
We code their returns to transverse sections and, using symbolic dynamics, find a ball of radius $O(e^{-CT/m})$ containing at most $m$ visits. Here, $m$ is a parameter that will be specified later. Thus the perturbation can be supported at a scale larger than
the smallest phase-space separation, which is of order $e^{-CT}$.
\smallskip

\item 
\emph{Separate the two lengths by a conformal perturbation.}
Use $g_s=(1+sf)^{-1}g$, as in Kaloshin--Wyman--Xi~\cite[Section~3.2]{KWX}.
In local graph coordinates, take $f$ proportional to
$\chi_r\prod_{i=1}^q|y-u_i(t)|^2$, where the $u_i$ are the unwanted
passages and $q\le m-1$. Both $f$ and $df$ vanish on one geodesic, so
its length stays fixed. The first variation of the other length is
$-\tfrac12\int f\,ds_g$. Normalization in $C^k$ gives a lower bound
$c_{k,m}e^{-C(m+k/m)T}$: the product costs $m$ in the exponent, and
the cutoff costs $k/m$.
\smallskip

\item 
\emph{Pass to fixed local Fourier directions.}
Choose all Fourier modes before the metric and geodesics. Their random
sum $u_\xi$ defines the conformal family $(1+u_\xi)^{-1}h$. For
prevalence, use its increment $v_\xi=((1+u_\xi)^{-1}-1)h$ as a fixed
random tensor. At every translated metric, Fourier approximation with
error $C_kN^{-k}$ detects a coordinate with nonzero length derivative.
\smallskip

\item 
\emph{Estimate small gaps and balance the two costs.}
A second-variation bound and a one-dimensional sublevel estimate give
$\mathbb P\{|\ell_\gamma-\ell_\eta|<a\}\le
C_{k,m,\eps}a e^{C(m+k/m)T}$.
Choose $m=\lceil\sqrt{k}\rceil$, count the exponentially many candidate
pairs, and apply Borel--Cantelli. The resulting gap exponent is
$O(\sqrt{k})$.
\item \emph{Obtain prevalence.}
The same probability estimate holds after every fixed tensor translation
whose probe lies in a bounded geometric neighborhood. Localization and a
countable sum of independent compact probes give one measure witnessing
shyness of the exceptional set.
\end{enumerate}

\medskip
\subsection*{Use of AI tools} ChatGPT (Astra 6) was used during manuscript revision to assist with exposition, symbolic calculations, and the checking and refinement of selected proof details. The authors take responsibility for the mathematical content and the final text.

\section{Separating two closed geodesics}
\label{geo:section}

We regard closed geodesics as unoriented: a free homotopy class and its inverse give the same label. This is the convention for the pairs $\gamma\ne\eta$ in the Main Theorem; without identifying the two orientations, their lengths would agree for every metric.

Fix the smooth negatively curved metric $h$ from the Main Theorem. Throughout the proof, $g$ belongs to a sufficiently small fixed $C^3$-neighborhood $\cU$ of $h$, and tensor norms are taken in a fixed finite smooth atlas. Constants denoted by $c,C$ depend only on this neighborhood and the atlas. Constants with subscripts may also depend on the indicated integers and on a bound $\norm{g}_{C^k}\le B$. Such dependence will occur in prefactors, but not in the coefficient of $T$ in an exponential estimate.

We write $S_hM:=\{(x,v)\in TM:h_x(v,v)=1\}$ for the unit tangent bundle of $h$. For each $g\in\cU$, radial normalization identifies $S_hM$ with $S_gM$ by $(x,v)\mapsto(x,v/\sqrt{g_x(v,v)})$. On this fixed bundle, $\varphi_t$ denotes the geodesic flow of $g$: it follows the geodesic with the prescribed initial position and direction for time $t$, measured in $g$-arclength, and records its new position and direction. The corresponding vector fields are uniformly $C^2$ close to that of $h$, so their flows are uniformly Anosov. Their speeds, bounded-time derivatives, and local transverse projections have uniform bounds.

We use the following form of expansivity: for every sufficiently small $a>0$, there is $e>0$, uniform in $\cU$, such that
\[
 \dist\bigl(\varphi_t(z),\varphi_{\theta(t)}(w)\bigr)<e\quad(t\in\mathbb R),\qquad \theta(0)=0,
\]
where $\theta$ is an increasing homeomorphism of $\mathbb R$, implies $w=\varphi_s(z)$ for some $|s|<a$; see~\cite{Bowen,BW}. Thus two trajectories that remain sufficiently close after an increasing change of time must be the same trajectory, with a small shift of the initial point.

We first fix transverse sections for these flows. The following standard choice follows from the construction of proper families in~\cite[Section~7]{Bowen}; see also~\cite[Section~3.2]{CLT}. In the smooth setting, the sections are chosen in smooth flow boxes. Compactness and transversality make the covering and return-time bounds uniform after decreasing $\cU$.

\begin{lemma}[A fixed family of transverse sections]
\label{geo:sections}
There are finitely many pairwise disjoint compact smooth transverse disks $D_i\subset S_hM$, compact cores $K_i\Subset D_i$, and constants $R,\tau>0$ with the following properties. Every orbit meets some $K_i$ in forward time at most $R$, and consecutive intersections with $\bigcup_iD_i$ are separated by at least $\tau$. Each $D_i$ has a slightly larger smooth transverse extension and a fixed flow box. These properties hold uniformly for $g\in\cU$. \qed
\end{lemma}

A \emph{flow orbit} through $z\in S_hM$ is the set $\{\varphi_t(z):t\in\mathbb R\}$. It is periodic if $\varphi_L(z)=z$ for some $L>0$; its least period is the length of the corresponding closed geodesic traversed once. The two orientations of a closed geodesic give two periodic flow orbits. Traversing the same oriented geodesic several times does not change its flow orbit. In the following lemma, which is the counterpart of~\cite[Appendix~A, Lemma~A.1]{DK}.

\begin{lemma}[Extreme neighbours lemma]
\label{geo:cluster}
\label{geo:phaseseparation}
Let $\Gamma$ be a nonempty union of finitely many periodic flow orbits whose total period, after removing repeated copies, is at most $C_0T$, where $T\ge1$ and $C_0$ is fixed. The constants may also depend on $C_0$. For every integer $m\ge1$, there are sections $E_i\Subset D_i$, containing the cores $K_i$, a nonempty subset $\mathcal C$ of $S:=\Gamma\cap\bigcup_iE_i$, and a point $z\in\mathcal C$ such that
\begin{equation}
 \#\mathcal C\le m,\qquad \dist_{E_i}\bigl(z,(S\cap E_i)\setminus\mathcal C\bigr)\ge\delta,\qquad \delta=c(1+T)^{-1}e^{-CT/m},
 \label{geo:sparseball}
\end{equation}
where $E_i$ contains $z$, and its distance is taken in the fixed disk coordinates. The ball $B_{E_i}(z,\delta)$ lies in the interior of $E_i$. We call the elements of $\mathcal C$ the \emph{extreme neighbours} of $z$, including $z$ itself. Moreover, distinct hits satisfy
\begin{equation}
 \dist(z,w)\ge c(1+T)^{-1}e^{-CT},\qquad z,w\in S,\quad z\ne w.
 \label{geo:hitsep}
\end{equation}
Two distinct oriented geometric periodic orbits of periods at most $T$ also have phase-space distance at least $c(1+T)^{-1}e^{-CT}$.
\end{lemma}

\begin{proof}
Let us first choose the boundaries of the sections away from the orbit hits. Inside each $D_i$, fix a smooth collar with level domains $E_i(a)=\{\rho_i\le a\}$, $1\le a\le2$, containing $K_i$ and staying a fixed distance from $\partial D_i$. By Lemma~\ref{geo:sections}, $\Gamma$ has at most $C(1+T)$ intersections with the full disks. We call these the virtual hits. Removing an interval of radius $c/(1+T)$ about each value $\rho_i(w)$ at a virtual hit leaves a choice of $a_i\in[1,2]$, provided $c$ is small. We take $E_i=E_i(a_i)$. Every virtual hit is then at distance at least $c/(1+T)$ from $\partial E_i$.

The first return $F$ to $\bigcup_iE_i$ permutes $S$, and $N:=\#S\le C(1+T)$. Its branch and return time at each hit are defined in a disk of radius $c/(1+T)$, with Lipschitz constants at most a fixed $A\ge2$. Indeed, the next hit occurs within a uniformly bounded time. The nearby trajectory has the same intersections with the sections: their inclusion is fixed by the boundary margin, and their order is fixed by the lower return-time bound $\tau$. The implicit function theorem in the fixed flow boxes gives the asserted Lipschitz estimates.

We now assign a finite alphabet to the hits. A letter records the section, a cell of a fine fixed coordinate grid, and a bin for the next return time. The spatial mesh and the width of the time bins are fixed, independently of $T$. Translating these grids puts their boundaries at distance at least $c/(1+T)$ from the finitely many recorded values. Thus the letter is constant in a disk of this radius about every hit.

These itineraries distinguish the points of $S$ when the meshes are sufficiently small. To see this, suppose two hits $z,w$ have the same future itinerary. Their periodicity gives the same bilateral itinerary. Matching consecutive return times by a piecewise affine time change makes the two trajectories uniformly close: the spatial and time errors on each return segment are bounded by a constant times the mesh size. We decrease $\tau$, if necessary, so that $a=\tau/2$ is an allowed choice in the expansivity statement above, and choose the meshes to make these errors smaller than its corresponding $e$. Expansivity gives $w=\varphi_s(z)$ with $|s|<\tau/2$. If $s\ne0$, this would give two distinct intersections of one orbit with the transverse disks in a time interval shorter than $\tau$, contrary to Lemma~\ref{geo:sections}. Hence $s=0$ and $w=z$.

Let $p(j)$ be the number of observed words of length $j$, with $p(0)=1$. We obtain
\begin{equation}
 p(j)\ge\min\{N,j+1\}.
 \label{geo:complexity}
\end{equation}
Indeed, if $p(j+1)=p(j)$, each observed $j$-word has a unique right extension. Since $F$ permutes $S$, successive extensions determine the entire future itinerary, and hence a single hit. Thus equality is possible only when $p(j)=N$; until then, $p(j)$ increases by at least one at every step.

If $N>m$, let $j=\lfloor N/(m+1)\rfloor$. By~\eqref{geo:complexity}, $p(j)>N/(m+1)$, so some word occurs at most $m$ times. Let $\mathcal C$ be its set of occurrences and choose $z\in\mathcal C$. A hit within $c(1+T)^{-1}A^{-j}$ of $z$ has the same first $j$ letters, because its distance from the corresponding iterates of $z$ stays below the boundary margin. It therefore belongs to $\mathcal C$. Since $j\le C(1+T)/m$, this proves~\eqref{geo:sparseball}, after decreasing $c$. When $N\le m$, we take $\mathcal C=S$ and use the boundary margin to choose the ball.

Let us also justify the separation of distinct hits. Put $\rho=c/(1+T)$, with $c$ small enough that every return branch is $A$-Lipschitz and every letter is constant in the radius-$\rho$ disk about a hit. Suppose two hits $z,w$ in the same section satisfy $\dist(z,w)<\rho A^{-N}$. Inductively, they follow the same return branches and
\[
 \dist\bigl(F^j(z),F^j(w)\bigr)\le A^j\dist(z,w)<\rho,\qquad 0\le j<N.
\]
Their first $N$ letters therefore agree. This contradicts $p(N-1)=N$, which says that words of length $N-1$ already distinguish all $N$ hits. Thus
\[
 \dist(z,w)\ge\rho A^{-N}\ge c(1+T)^{-1}e^{-CT},
\]
since $N\le C(1+T)$ and $A$ is fixed. Distinct compact disks $D_i$ have a fixed positive mutual distance, so hits in different disks satisfy the same bound after decreasing $c$. This proves~\eqref{geo:hitsep}.

Finally, consider two distinct oriented periodic orbits and use the adaptive sections for their union. Flow a point on the first orbit to a core in time at most $R$. A sufficiently close point on the second orbit gives, by bounded-time tracking and transverse projection, a distinct hit at distance at most a fixed multiple of the original distance. The boundary margin keeps this hit in the same adaptive section. Estimate~\eqref{geo:hitsep} therefore gives the claimed phase-space separation.
\end{proof}

The last assertion concerns distinct geometric flow orbits. Nearby points on the same orbit are compared only after transverse projection. The two orientations of a closed geodesic give distinct flow orbits and are both included in the next construction.

We will need one local consequence of the geodesic equation. It converts separation in position and direction into a bound on the portion of an arc where the positions can be close.

\begin{lemma}[Short geodesic graphs]
\label{geo:graphs}
On a fixed sufficiently short interval $I$, let two nearby geodesic arcs be written as graphs in a smooth chart, with bounded slopes, and let $b:I\to\mathbb R^{\dim M-1}$ be their difference. Then
\begin{equation}
 b''=U(t)b+V(t)b',\qquad \norm{U}_{C^0}+\norm{V}_{C^0}\le C.
 \label{geo:graphode}
\end{equation}
If $|b(t)|+|b'(t)|\ge q$ on $I$ and $0<a<cq$, then
\begin{equation}
 \bigl|\{t\in I:|b(t)|<a\}\bigr|\le Ca/q.
 \label{geo:sublevel}
\end{equation}
The interval length and the constants are uniform in $\cU$.
\end{lemma}

\begin{proof}
In coordinates $(t,y)$, a geodesic graph satisfies
\[
 (u^\alpha)''=-\Gamma^\alpha_{ij}(t,u)w^iw^j+(u^\alpha)'\Gamma^1_{ij}(t,u)w^iw^j,\qquad w=(1,u').
\]
Subtracting the two equations and using the mean value formula gives~\eqref{geo:graphode}. Let us write out the resulting comparison. For $s\le t$ in $I$, integration gives
\[
 |b(t)|+|b'(t)|\le |b(s)|+|b'(s)|+C\int_s^t\bigl(|b(v)|+|b'(v)|\bigr)\,dv.
\]
Gronwall's inequality bounds the right side by $e^{C(t-s)}(|b(s)|+|b'(s)|)$. Applying the same argument backwards in time, we obtain
\begin{equation}
 \begin{aligned}
 e^{-C|t-s|}\bigl(|b(s)|+|b'(s)|\bigr)
 &\le |b(t)|+|b'(t)|\\
 &\le e^{C|t-s|}\bigl(|b(s)|+|b'(s)|\bigr),\qquad s,t\in I.
 \end{aligned}
 \label{geo:gronwall}
\end{equation}
Since $I$ has fixed length, the comparison factor is uniform.

If the set in~\eqref{geo:sublevel} is nonempty, let $t_0$ belong to it. Then $|b(t_0)|<a<cq$ and $|b'(t_0)|\ge q-a$. In particular, for $c\le1/2$, we have $|b(t_0)|+|b'(t_0)|\le2|b'(t_0)|$. Integrating~\eqref{geo:graphode} and using~\eqref{geo:gronwall}, we get
\[
 \sup_{t\in I}|b'(t)-b'(t_0)|
 \le 2C|I|e^{C|I|}|b'(t_0)|
 \le\tfrac12|b'(t_0)|,
\]
where the last inequality follows by fixing $|I|$ sufficiently small.
Consequently the projection of $b(t)$ in the direction $b'(t_0)$ is strictly increasing, with derivative at least $|b'(t_0)|/2$. The preimage of $[-a,a]$ has length at most $4a/|b'(t_0)|\le Ca/q$, as required.
\end{proof}

We use the conformal perturbation in~\cite[Section~3.2]{KWX},
\begin{equation}
 g_s=(1+sf)^{-1}g,\qquad p_{g_s}=(1+sf)p_g,
 \label{geo:conformal}
\end{equation}
where $p_g(x,\xi)=g^{ij}(x)\xi_i\xi_j$ and $|s|\norm{f}_{C^0}<1$. The function $f$ will be chosen to separate two lengths.

\begin{proposition}[A separating metric variation]
\label{geo:detector}
Let $k\ge3$, $\norm{g}_{C^k}\le B$, and let $\gamma\ne\eta$ be closed geodesics, with repetitions allowed, of lengths at most $T\ge1$. For every integer $m\ge1$, there is $f\in C^k(M)$ such that
\begin{equation}
 \norm{f}_{C^k}\le1,\qquad \left|\left.\frac{d}{ds}\bigl(\ell_{g_s}(\gamma_s)-\ell_{g_s}(\eta_s)\bigr)\right|_{s=0}\right|\ge c_{k,m,B}e^{-C(m+k/m)T}.
 \label{geo:detectorbound}
\end{equation}
Here $g_s$ is given by~\eqref{geo:conformal}, and $\gamma_s,\eta_s$ are the geodesics in the corresponding classes. If their geometric images are distinct, $f$ can be chosen nonnegative, supported in one coordinate chart, with $f=df=0$ on one of them.
\end{proposition}

\begin{proof}
Suppose first that the geometric geodesics are distinct. Lemma~\ref{geo:cluster}, applied to both orientations of both geodesics, gives a hit $z$ with at most $m$ extreme neighbours at scale $\delta$. After interchanging $\gamma,\eta$ and reversing parameters if necessary, we may assume that $z$ belongs to $\gamma$.

In a fixed smooth chart, followed by a bounded linear coordinate change, write a short arc of $\gamma$ as $x(t)=(t,u_0(t))$, $t\in I$, with $x(0)=\pi(z)$. Here $\pi:S_hM\to M$ is the bundle projection, and $I$ has a fixed small length. Initial data $(u(0),u'(0))$ of nearby geodesic graphs are uniformly bi-Lipschitz equivalent to their hits in the transverse disk containing $z$. Consequently, at most $m-1$ distinct graphs of $\eta$ have initial data within $c\delta$ of those of $u_0$: each gives an extreme neighbour other than $z$. Denote them by $u_1,\ldots,u_q$, where $q\le m-1$. We count coincident passages only once. Including both orientations accounts for passages whose original direction is opposite to that of $\gamma$.

Let us find a point of the main arc away from all other passages of $\eta$. Divide $\eta$ into at most $C(1+T)$ short arcs, with overlapping extensions. An arc approaching $x(I)$ with a small angle in either orientation extends as a graph over $I$. If it is not one of the selected graphs, its initial-data distance is at least $c\delta$. By~\eqref{geo:gronwall},
\[
 |u(t)-u_0(t)|+|u'(t)-u_0'(t)|\ge c\delta,\qquad t\in I.
\]
Lemma~\ref{geo:graphs} then bounds the portion of $I$ at base distance less than $4r$ from this arc by $Cr/\delta$. An arc making a fixed positive angle has the stronger bound $Cr$. To see this, parametrize it by $g$-arclength as $(t(s),y(s))$ and put $w(s)=y(s)-u_0(t(s))$. At a point close to the main arc, the angle assumption gives $|w'(s_0)|\ge c$. With $e=w'(s_0)/|w'(s_0)|$, the uniform bounds on second derivatives imply $\langle w'(s),e\rangle\ge c/2$ on the sufficiently short arc. Points at distance at most $4r$ from $x(I)$ satisfy $|w(s)|\le Cr$. Therefore their parameters belong to the interval where $|\langle w(s),e\rangle|\le Cr$, whose length is at most $4Cr/c$. This is precisely a bound on the arclength of the portion near $x(I)$. Since $|t'(s)|\le C$, its projection onto $I$, and the $Cr$-neighborhood of that projection, also have length at most $Cr$.

We call these unselected passages the far arcs and choose
\begin{equation}
 r=c\delta/(1+T).
 \label{geo:radius}
\end{equation}
The part of $I$ within distance $4r$ of the far arcs has total length at most $C(1+T)r/\delta$. Choosing $c$ small makes this less than $|I|/8$.

We must also separate the positions of the selected graphs. Put $\sigma=c(1+T)^{-1}e^{-CT}$. The phase-space separation in Lemma~\ref{geo:cluster} gives $|u_i-u_0|+|u_i'-u_0'|\ge c\sigma$ on $I$. With $a=c\sigma/m$, Lemma~\ref{geo:graphs} shows that the union of the sets $\{|u_i-u_0|<a\}$ has length at most $Cma/\sigma<|I|/8$, after decreasing $c$. Hence the middle half of $I$ contains $t_*$ outside both exceptional sets. The ball of radius $2r$ about $x(t_*)$ misses all far arcs, and
\begin{equation}
 |u_i(t_*)-u_0(t_*)|\ge a,\qquad 1\le i\le q.
 \label{geo:goodpoint}
\end{equation}

Choose a nonnegative smooth bump $\chi_r$, equal to one on the ball of radius $r$ about $x(t_*)$ and supported in the ball of radius $2r$. In the chart, put
\begin{equation}
 F(t,y)=\chi_r(t,y)\prod_{i=1}^q|y-u_i(t)|^2,
 \label{geo:product}
\end{equation}
and extend it by zero. The empty product is one. Every passage of $\eta$ through the support lies on a selected graph. The squared factors therefore give $F|_\eta=0$ and $dF|_\eta=0$.

The graph equation bounds $\norm{u_i}_{C^k}$ by a constant depending only on $k,B$: successive differentiation uses derivatives of the metric of order at most $k-1$. Since the coordinates were fixed smooth charts, the construction loses no regularity. The product rule gives
\begin{equation}
 \norm{F}_{C^k}\le C_{k,m,B}r^{-k}.
 \label{geo:productnorm}
\end{equation}
Furthermore, the first derivatives of all graphs have a uniform bound. Thus~\eqref{geo:goodpoint}, with $a/2$ in place of $a$, holds on an interval of length $c\min\{r,a\}$ about $t_*$. On this interval $\chi_r=1$, so
\begin{equation}
 \int_\gamma F\,ds_g\ge c_m\min\{r,a\}a^{2q}.
 \label{geo:positiveintegral}
\end{equation}
We normalize by setting $f=F/(C_{k,m,B}r^{-k})$, with the denominator from~\eqref{geo:productnorm}.

To compute the length derivative, recall the first-variation formula
\begin{equation}
 \left.\frac{d}{ds}\ell_{g_s}(\gamma_s)\right|_{s=0}=\frac12\int_\gamma \left.\partial_sg_s\right|_{s=0}(\dot\gamma,\dot\gamma)\,ds_g,
 \label{geo:firstvariation}
\end{equation}
where $\gamma$ is parametrized by $g$-arclength. The curve-variation term vanishes because $\gamma$ is a closed geodesic. Negative curvature makes the normal periodic Jacobi operator invertible, giving the local continuation in the metric; tangential changes only reparametrize the curve.

For~\eqref{geo:conformal}, the derivative is $-\tfrac12\int_\gamma f\,ds_g$. Also, $f=df=0$ on $\eta$, so the first jets of $g_s$ and $g$ agree there. The curve $\eta$ remains a geodesic of the same length. It follows from~\eqref{geo:positiveintegral} that the absolute derivative of the difference is at least
\[
 c_{k,m,B}r^k\min\{r,a\}a^{2q}.
\]
By~\eqref{geo:sparseball} and~\eqref{geo:radius},
\[
 r\ge c(1+T)^{-2}e^{-CT/m},\qquad a\ge c_m(1+T)^{-1}e^{-CT},\qquad \min\{r,a\}\ge c_m(1+T)^{-2}e^{-CT}.
\]
Since $q\le m-1$, the resulting lower bound is
\[
 c_{k,m,B}(1+T)^{-2k-2m}\exp\left[-C\left(\frac{k}{m}+2m\right)T\right].
\]
The polynomial factor has the required exponential cost as well. Indeed, $\log(1+T)\le T/m+\log(1+m)$ gives
\[
 (1+T)^{-2k}\ge(1+m)^{-2k}e^{-2kT/m},\qquad (1+T)^{-2m}\ge e^{-2mT}.
\]
Absorbing $(1+m)^{-2k}$ into the prefactor proves~\eqref{geo:detectorbound}. The same bound holds for repeated traversals. Indeed, if $\gamma$ traverses its geometric geodesic $\gamma_0$ exactly $p$ times, then
\[
 \int_\gamma f\,ds_g=p\int_{\gamma_0}f\,ds_g\ge\int_{\gamma_0}f\,ds_g,
\]
since $p\ge1$ and $f\ge0$. The integral along $\eta$ remains zero, regardless of its traversal number. Thus repetitions can only increase the nonzero first variation.

It remains to consider different traversal numbers of the same geometric geodesic $\gamma_0$. We take $f$ to be a fixed positive constant with $\norm{f}_{C^k}\le1$. If the traversal numbers are $p\ne q$, then
\[
 \ell_{g_s}(\gamma_s)-\ell_{g_s}(\eta_s)=(p-q)(1+sf)^{-1/2}\ell_g(\gamma_0),
\]
so the absolute derivative at zero is $\tfrac12 f|p-q|\ell_g(\gamma_0)\ge\tfrac12 f\ell_g(\gamma_0)$.

Here we use the \emph{systole} $\operatorname{sys}(g)$, the length of a shortest nontrivial closed geodesic of $g$. It is bounded below uniformly on $\cU$. Indeed, after shrinking $\cU$, there is $c_0>0$ such that $g\ge c_0h$ as quadratic forms. Every curve then has $g$-length at least $\sqrt{c_0}$ times its $h$-length. Minimizing in each free homotopy class and then over the classes gives
\begin{equation}
 \ell_g(\alpha)\ge\sqrt{c_0}\ell_h(\alpha),\qquad
 \operatorname{sys}(g)\ge\sqrt{c_0}\operatorname{sys}(h)>0.
 \label{geo:systole}
\end{equation}
The last inequality follows from the positive injectivity radius of the fixed compact manifold $(M,h)$. Thus the derivative above is at least $\tfrac12 f\sqrt{c_0}\operatorname{sys}(h)$, proving the remaining case.
\end{proof}

\begin{corollary}
\label{geo:sqrt}
For every $k\ge3$, the variation in Proposition~\ref{geo:detector} can be chosen with
\[
 \norm{f}_{C^k}\le1,\qquad \left|\left.\frac{d}{ds}\bigl(\ell_{g_s}(\gamma_s)-\ell_{g_s}(\eta_s)\bigr)\right|_{s=0}\right|\ge c_{k,B}e^{-C\sqrt{k}\,T}.
\]
\end{corollary}

\begin{proof}
The choice $m=\lceil\sqrt{k}\rceil$ gives $m+k/m\le2\sqrt{k}+1$ in Proposition~\ref{geo:detector}.
\end{proof}

\begin{remark}
Local conformal perturbations also occur in Schenck's construction~\cite{Schenck}. Here, the product in~\eqref{geo:product} removes at most $m-1$ nearby passages exactly, while the cutoff is placed at the larger scale $e^{-CT/m}$. Their respective costs are $m$ and $k/m$ in the exponent. This is the geometric counterpart of the product and localization used in~\cite[Appendix~A, proof of Proposition~A.3]{DK}.
\end{remark}

\section{Atlas of local Fourier perturbations}
\label{prob:section}

We now pass from the perturbation in Proposition~\ref{geo:detector} to a fixed probability measure on the space of metrics. Throughout this section, $n=\dim M$, $k\ge3$, and all metrics belong to the fixed $C^3$-neighborhood of $h$. We restrict to bounded $C^k$-families when a uniform high-order bound is needed. Constants depending on that bound may enter the prefactors, but not the coefficient of $\sqrt{k}\,T$ in the exponent.

First, let us express the separating direction as a scalar multiple of the fixed metric $h$. Consider two distinct geometric closed geodesics, with repetitions allowed, whose lengths are at most $T$. Proposition~\ref{geo:detector} supplies, after interchanging the pair if necessary, a nonnegative function $f$ vanishing on $\eta$, with $\norm{f}_{C^k}\le1$. Since $h(\dot\gamma,\dot\gamma)$ has a uniform positive lower bound for $g$-unit vectors, its proof gives
\begin{equation}
 \left|\frac12\int_\gamma f h(\dot\gamma,\dot\gamma)\,ds_g-\frac12\int_\eta f h(\dot\eta,\dot\eta)\,ds_g\right|\ge a_{k,m}\exp\left[-C_h\left(m+\frac{k}{m}\right)T\right].
 \label{prob:detector}
\end{equation}
The prefactor is uniform on bounded $C^k$-families. Indeed, the second integral vanishes, and the first is bounded below by a fixed positive multiple of $\frac12\int_\gamma f\,ds_g$. Thus the directions $fh$ separate the lengths at every metric in the neighborhood, including after arbitrary tensor translations.

\begin{lemma}[A fixed finite Fourier approximation]
\label{prob:jackson}
There is a fixed countable collection of smooth real functions $e_j$ on $M$, with frequencies $\nu_j\ge1$, such that
\begin{equation}
 \#\{j:\nu_j\le N\}\le C N^n,\qquad \norm{e_j}_{C^a}\le C_a\nu_j^a.
 \label{prob:modes}
\end{equation}
For every $f\in C^k(M)$ and $N\ge2$, there is a finite combination $f_N=\sum_{\nu_j\le CN}b_j e_j$ satisfying
\begin{equation}
 |b_j|\le C\norm{f}_{C^0},\qquad \norm{f-f_N}_{C^0}\le C_k N^{-k}\norm{f}_{C^k}.
 \label{prob:jackson-estimate}
\end{equation}
\end{lemma}

\begin{proof}
We first prove the approximation estimate on a flat torus. Choose an even function $\Phi\in C_c^\infty(\mathbb R^n)$ equal to one near zero, with $0\le\Phi\le1$, and put
\[
 K(x)=(2\pi)^{-n}\int_{\mathbb R^n}e^{i x\cdot\xi}\Phi(\xi)\,d\xi.
\]
Repeated integration by parts shows that $K$ and its derivatives decay faster than every inverse power of $|x|$. Fourier inversion gives $\int K(x)\,dx=1$ and $\int x^\alpha K(x)\,dx=0$ for every multi-index with $|\alpha|\ge1$, because all positive-order derivatives of $\Phi$ vanish at zero. These are the standard Fourier facts used below; see~\cite[Chapter~I, Section~1]{Stein}.

For a $2\pi$-periodic function $u$, define
\[
 S_Nu(x)=\int_{\mathbb R^n}N^nK(Nz)u(x-z)\,dz=\sum_{q\in\mathbb Z^n}\Phi(q/N)\widehat u(q)e^{i q\cdot x}.
\]
The sum is finite. Taylor's formula to order $k-1$ and the moment identities cancel every term except the remainder. Consequently,
\[
 |S_Nu(x)-u(x)|\le C_k\norm{u}_{C^k}\int N^n|K(Nz)|\,|z|^k\,dz\le C_kN^{-k}\norm{u}_{C^k}.
\]
Also, $|\Phi(q/N)\widehat u(q)|\le\norm{u}_{C^0}$. The smooth frequency cutoff therefore gives the full approximation rate $N^{-k}$, without losing derivatives.

We now choose a finite smooth atlas, a subordinate partition of unity $\rho_i$, and cutoffs $\chi_i$ supported in the corresponding charts and equal to one on $\operatorname{supp}\rho_i$. After rescaling the coordinates, each $\rho_i f$, extended by zero, is a $C^k$ function on a flat torus. Let $u_i$ denote this extension and set $f_N=\sum_i\chi_i S_Nu_i$ in the charts. Since $\sum_i\chi_i u_i=f$, the preceding estimate gives the required uniform error. The fixed modes are $\chi_i\cos(q\cdot x)$ and $\chi_i\sin(q\cdot x)$, including the modes $\chi_i$ for $q=0$. With $\nu_j=1+|q|$, counting lattice points and differentiating these functions proves~\eqref{prob:modes}; the Fourier coefficient bound proves the other assertion in~\eqref{prob:jackson-estimate}. All charts, cutoffs, and modes are chosen before the metric and the geodesic pair.
\end{proof}

The modes depend on the atlas and cutoffs, and no orthogonality is needed. Their role is the same as that of the fixed circle modes in~\cite[Appendix A, Proposition~A.3]{DK}: a perturbation adapted to the two orbits will be replaced by one direction from a list fixed in advance.

Let $\xi_j$ be independent random variables uniformly distributed on $[-1,1]$. For $0<\eps\le\eps_0$, put
\begin{equation}
 u_\xi=\eps\sum_j\nu_j^{-(k+n+2)}\xi_j e_j,\qquad v_\xi=\bigl((1+u_\xi)^{-1}-1\bigr)h.
 \label{prob:probe}
\end{equation}
In particular, $h+v_\xi=(1+u_\xi)^{-1}h$ and $p_{h+v_\xi}=(1+u_\xi)p_h$, as in~\eqref{geo:conformal}. This also holds conditionally in each coordinate. Indeed, writing $\widehat u=\eps\sum_{i\ne j}\nu_i^{-(k+n+2)}\xi_i e_i$ and $H=(1+\widehat u)^{-1}h$, we obtain
\[
 h+v_{(\widehat\xi,t)}=\left(1+t\frac{\eps\nu_j^{-(k+n+2)}e_j}{1+\widehat u}\right)^{-1}H.
\]
The varying scalar factor is supported where $e_j$ is supported.

Let us verify that~\eqref{prob:probe} defines a compactly supported measure in $C^k(S^2T^*M)$. By~\eqref{prob:modes}, its scalar series satisfies
\[
 \sup_\xi\norm{u_\xi}_{C^k}\le C_k\eps\sum_j\nu_j^{-n-2}<\infty,\qquad \sup_\xi\left\|\eps\sum_{\nu_j>N}\nu_j^{-(k+n+2)}\xi_j e_j\right\|_{C^k}\longrightarrow0.
\]
The summability follows, for example, by grouping frequencies in the intervals $[2^l,2^{l+1})$: the corresponding sum is at most $C2^{-2l}$. Finite partial sums are continuous functions of finitely many parameters. The uniform tail bound makes the full series a continuous map from the compact product $[-1,1]^{\mathbb N}$ into $C^k(M)$, so its image is compact. The $C^0$ bound is uniform in $k\ge3$; hence $\eps_0>0$ can be chosen independently of $k$ so that $\norm{u_\xi}_{C^0}\le1/2$. Continuity of $u\mapsto((1+u)^{-1}-1)h$ now proves compactness of the support of the law of $v_\xi$. Decreasing $\eps$ places this support in any prescribed $C^k$-ball about zero.

\begin{remark}
Random perturbations parametrized by Hilbert bricks were used by Kaloshin--Hunt~\cite[Sections~1.3 and~2.3]{KH} to prove prevalent stretched exponential bounds for the number of periodic points of diffeomorphisms. Here, the summable Fourier weights provide a compact parameter image in finite regularity, as in the circle construction of~\cite[Appendix A.3]{DK}.
\end{remark}

For prevalence we use the same law after every additive tensor translation. The resulting family is $a+v_\xi$, which need not be conformal to $a$. This causes no difficulty: the definition of prevalence requires one fixed measure, not an affine parametrization of its support.

\begin{lemma}[Second variation along a metric path]
\label{prob:second-variation}
Let $g_t$ be a twice differentiable path of $C^k$ metrics in the fixed geometric neighborhood. Then
\begin{equation}
 \left|\frac{d^2}{dt^2}\ell_{g_t}(\gamma)\right|\le C_h\ell_{g_t}(\gamma)\bigl(\norm{\partial_tg_t}_{C^1}^2+\norm{\partial_t^2g_t}_{C^0}\bigr).
 \label{prob:hessian}
\end{equation}
The same estimate holds for iterated closed geodesics.
\end{lemma}

\begin{proof}
It suffices to calculate at $t=0$. Put $g=g_0$, $q=\partial_tg_t|_0$, and $b=\partial_t^2g_t|_0$. Parametrize the closed geodesic by $g$-arclength, write $v=\dot\gamma$, and choose the parametrizations of its continuations so that their variation $V$ is normal to $v$. The classical second-variation formula gives the index form on periodic normal fields; see~\cite[Chapter~XI, Section~1]{Lang}. In our notation, it is
\[
 I(V,W)=\int_\gamma\bigl(\langle D_sV,D_sW\rangle-\langle R(V,v)v,W\rangle\bigr)\,ds.
\]
If $-\kappa^2<0$ is a uniform upper bound for the sectional curvatures in $\cU$, then
\begin{equation}
 I(V,V)\ge\norm{D_sV}_{L^2}^2+\kappa^2\norm{V}_{L^2}^2.
 \label{prob:coercivity}
\end{equation}
Thus the index form controls both the field $V$ and its covariant derivative along the geodesic. This positive lower bound is what we mean by \emph{coercivity}; its constant does not depend on the period. In particular, $\norm{V}_{L^2}\le\kappa^{-1}I(V,V)^{1/2}$.

Differentiating the first metric variation of length in the curve direction $W$ and integrating by parts, we obtain
\[
 B_q(W)=\frac12\int_\gamma(\nabla_Wq)(v,v)\,ds-\int_\gamma(\nabla_vq)(W,v)\,ds,\qquad |B_q(W)|\le C_h\sqrt{\ell_g(\gamma)}\norm{q}_{C^1}\norm{W}_{L^2}.
\]
Here the arclength variation vanishes because $W$ is normal and $\gamma$ is a geodesic. Differentiation of the critical-point equation gives $I(V,W)+B_q(W)=0$. Taking $W=V$ and using~\eqref{prob:coercivity}, we get
\[
 I(V,V)\le C_h\sqrt{\ell_g(\gamma)}\norm{q}_{C^1}\norm{V}_{L^2}\le C_h\sqrt{\ell_g(\gamma)}\norm{q}_{C^1}I(V,V)^{1/2}.
\]
Thus $I(V,V)\le C_h\ell_g(\gamma)\norm{q}_{C^1}^2$. For the critical value of length, the second derivative is
\[
 \left.\frac{d^2}{dt^2}\ell_{g_t}(\gamma)\right|_{t=0}=\frac12\int_\gamma b(v,v)\,ds-\frac14\int_\gamma q(v,v)^2\,ds-I(V,V).
\]
The preceding bound proves~\eqref{prob:hessian}. Estimate~\eqref{prob:coercivity} also justifies the continuation used in the calculation. Indeed, a periodic normal Jacobi field lies in the kernel of the index form, so $I(V,V)=0$, and~\eqref{prob:coercivity} forces $V=0$. The periodic normal Jacobi operator is self-adjoint and elliptic on the compact parameter circle; having zero kernel, it is invertible. The implicit function theorem therefore gives the required twice differentiable continuation for $C^k$ metrics, $k\ge3$, after removing the freedom of parametrization. For a multiply covered geodesic, the same argument uses its full period and the same constants.
\end{proof}

The following sublevel estimate is precisely the one used in the circle argument; see~\cite[Appendix A, Lemma~A.4]{DK}. The coordinate with a large derivative is allowed to depend on the point.

\begin{lemma}[A sublevel estimate]
\label{prob:small-values}
Let $F\in C^2([-1,1]^D)$, and suppose that, at every point,
\[
 \max_{1\le i\le D}|\partial_iF|\ge b>0,\qquad |\partial_i^2F|\le H\quad(1\le i\le D).
\]
For normalized Lebesgue measure and every $a>0$,
\[
 \mathbb P\{|F|<a\}\le2D\frac{a}{b}\left(1+\frac{4H}{b}\right).
\] \qed
\end{lemma}

\begin{proposition}[A uniform probability estimate]
\label{prob:pair}
Suppose a translated family $g+v_\xi$ from~\eqref{prob:probe} is entirely contained in a bounded $C^k$-family in the geometric neighborhood. Fix two different geometric closed geodesics, allowing repetitions, whose lengths are at most $T\ge1$ for every parameter. Then
\begin{equation}
 \mathbb P\left\{|\ell_{g+v_\xi}(\gamma)-\ell_{g+v_\xi}(\eta)|<a\right\}\le C_{k,\eps}a\exp(C_h\sqrt{k}\,T),\qquad a>0.
 \label{prob:pair-bound}
\end{equation}
The prefactor may depend on the $C^k$ bound. The exponent constant depends only on the fixed geometric neighborhood, dimension, and the constants in Proposition~\ref{geo:detector}.
\end{proposition}

\begin{proof}
We first retain $m$ as a parameter. Denote the lower bound in~\eqref{prob:detector} by $A_T$, and let $f$ be a separating function at the metric $g+v_\xi$. In a scalar direction $v$, the derivative of the tensor $v_\xi$ is $-(1+u_\xi)^{-2}vh$. The first variation formula~\eqref{geo:firstvariation} therefore gives the following linear functional of $v$:
\[
 -\frac12\int_\gamma\frac{v}{(1+u_\xi)^2}h(\dot\gamma,\dot\gamma)\,ds_{g+v_\xi}+\frac12\int_\eta\frac{v}{(1+u_\xi)^2}h(\dot\eta,\dot\eta)\,ds_{g+v_\xi}.
\]
Its norm on $C^0(M)$ is at most $C_hT$. To recover the direction $fh$, we take $\widehat f=-(1+u_\xi)^2f$, whose $C^k$ norm is uniformly bounded by $C_k$ on the parameter cube.

We approximate $\widehat f$ by the fixed modes from Lemma~\ref{prob:jackson}. Choose the least integer $N\ge2$ such that
\begin{equation}
 N^k\ge C_kT/A_T.
 \label{prob:cutoff}
\end{equation}
Then $\widehat f_N=\sum_{\nu_j\le CN}b_je_j$ has $|b_j|\le C_k$, and its value under the displayed functional has modulus at least $A_T/2$. Indeed, the approximation changes that value by at most $C_hT\norm{\widehat f-\widehat f_N}_{C^0}\le A_T/2$, after increasing the constant in~\eqref{prob:cutoff}. There are at most $CN^n$ terms, so at least one direction $-(1+u_\xi)^{-2}e_jh$ gives a length derivative of modulus at least $A_T/(C_kN^n)$. Including its weight in~\eqref{prob:probe}, we obtain
\begin{equation}
 |\partial_{\xi_j}(\ell_\gamma-\ell_\eta)|\ge b_T:=c_k\eps A_TN^{-(k+2n+2)},\qquad \nu_j\le CN.
 \label{prob:coordinate}
\end{equation}
This choice is made separately at each parameter. We do not differentiate the separating function or its Fourier coefficients. Although the selected index can vary, its finite list and the lower bound $b_T$ are fixed throughout the parameter cube.

Let us now bound the second derivatives with respect to these coordinates. Direct differentiation gives
\[
 \partial_{\xi_j}v_\xi=-\eps\nu_j^{-(k+n+2)}\frac{e_jh}{(1+u_\xi)^2},\qquad \partial_{\xi_j}^2v_\xi=2\eps^2\nu_j^{-2(k+n+2)}\frac{e_j^2h}{(1+u_\xi)^3}.
\]
By~\eqref{prob:modes}, $\norm{u_\xi}_{C^1}\le C_h\eps$ uniformly in $k\ge3$. Since $1+u_\xi\ge1/2$, the product rule gives $\norm{\partial_{\xi_j}v_\xi}_{C^1}\le C_h\eps$ and $\norm{\partial_{\xi_j}^2v_\xi}_{C^0}\le C_h\eps^2$. Consequently, for $F(\xi)=\ell_{g+v_\xi}(\gamma)-\ell_{g+v_\xi}(\eta)$, Lemma~\ref{prob:second-variation} implies
\[
 |\partial_{\xi_j}^2F|\le C_hT\eps^2.
\]
The acceleration term in~\eqref{prob:hessian} accounts for the nonlinear parametrization of the probe.

We condition on all coordinates with $\nu_j>CN$. For every fixed tail, $F$ is a $C^2$ function of at most $CN^n$ remaining coordinates. Its maximal first derivative is at least $b_T$ by~\eqref{prob:coordinate}, and the preceding second-derivative bound is uniform in the tail. Lemma~\ref{prob:small-values}, followed by integration over the tail, yields
\[
 \mathbb P\{|F|<a\}\le CN^n\frac{a}{b_T}\left(1+\frac{C_hT\eps^2}{b_T}\right).
\]
In particular, no measurable choice of the separating coordinate is needed.

It remains to keep track of the exponential cost. Put $X=m+k/m$, so that $A_T=a_{k,m}e^{-C_hXT}$. The choice~\eqref{prob:cutoff} and the definition of $b_T$ give
\[
 N^n\le C_{k,m}T^{n/k}e^{C_hXnT/k},\qquad b_T^{-1}\le C_{k,m,\eps}T^{1+(2n+2)/k}e^{C_hX(2+(2n+2)/k)T}.
\]
Combining these bounds with the preceding probability estimate, and enlarging the prefactor to include the term containing only $b_T^{-1}$, we obtain
\[
 \mathbb P\{|F|<a\}\le C_{k,m,\eps}a\,T^{3+(5n+4)/k}e^{C_hX(4+(5n+4)/k)T}.
\]
For $k\ge3$, the coefficients involving $1/k$ are bounded in terms of the dimension alone. The polynomial in $T$ can also be absorbed into a fixed additional exponential because $T\ge1$ and $X\ge2$. Thus the last expression is at most $C_{k,m,\eps}a e^{C_hXT}$, after increasing $C_h$ independently of $k$ and $m$. Finally, taking $m=\lceil\sqrt{k}\rceil$, as in Corollary~\ref{geo:sqrt}, gives $X\le2\sqrt{k}+1$ and proves~\eqref{prob:pair-bound}.
\end{proof}

\begin{proposition}[Almost sure exponential gaps]
\label{prob:gaps}
There is $B=B(h)>0$, independent of $k\ge3$, such that almost every metric in every translated probe family as in Proposition~\ref{prob:pair} satisfies
\begin{equation}
 |\ell_g(\gamma)-\ell_g(\eta)|\ge c(g)\exp\left[-B\sqrt{k}\max\{\ell_g(\gamma),\ell_g(\eta)\}\right]
 \label{prob:gap-property}
\end{equation}
for all different unoriented free homotopy classes, with some $c(g)>0$.
\end{proposition}

\begin{proof}
In this proof, $C,H,B$ depend on the fixed metric $h$ and its chosen geometric neighborhood, but not on $k$. We first count the classes that can occur below a given length. There are at most $Ce^{HT}$ classes of $h$-length at most $T$. Indeed, choosing a compact fundamental domain in the universal cover, a closed geodesic of length at most $T$ gives a deck transformation, after conjugation, moving a fixed basepoint by at most $T+C$. Cocompactness allows us to divide a joining geodesic into bounded segments and choose nearby orbit points. Successive orbit points differ by an element of a fixed finite set, so the deck transformation is a word of length at most $C(T+1)$ in a fixed finite alphabet. Counting words proves the asserted upper bound, including iterates.

Fix an integer $N\ge1$. Uniform equivalence with $h$ implies that a class with length at most $N$ at one parameter has $h$-length at most $CN$, and hence length at most $C^2N$ at every parameter. There are at most $Ce^{HN}$ such candidate classes, after changing $H$. Proposition~\ref{prob:pair}, with $T=C^2N$, and the union bound therefore give
\[
 \mathbb P\left\{\begin{array}{c}\text{some two classes of lengths at most }N\\\text{have a gap smaller than }e^{-\beta N}\end{array}\right\}\le C_{k,\eps}\exp\bigl[(C\sqrt{k}+2H-\beta)N\bigr]
\]
for all sufficiently large $N$. Pairs of different powers of the same geometric geodesic do not contribute once $e^{-\beta N}$ is below the uniform systole bound~\eqref{geo:systole}.

We choose $B>C+2H+1$ and set $\beta=B\sqrt{k}$. The probabilities are summable. By the first Borel--Cantelli lemma, almost surely every pair of length at most $N$ has gap at least $e^{-\beta N}$ for all sufficiently large integers $N$.

We also need to exclude exact equalities among the finitely many shorter geodesics. For any fixed pair of different geometric geodesics, the two lengths are bounded on the compact probe cube. Letting $a\downarrow0$ in~\eqref{prob:pair-bound} shows that their equality has probability zero. There are countably many pairs, so almost surely none of these equalities occurs. The remaining pairs below the random cutoff from Borel--Cantelli are finite in number and have positive gaps. Reducing $c(g)>0$ includes them. For a pair with maximum length $L$, the choice $N=\lceil L\rceil$ gives $e^{-\beta N}\ge e^{-\beta}e^{-\beta L}$, which proves~\eqref{prob:gap-property}.
\end{proof}

\section{Prevalence}
\label{prob:prevalence-section}

\begin{proof}[Proof of Main Theorem]
Let us first identify the correspondence with the circle argument in~\cite[Appendix A]{DK}. The transverse separation estimate~\eqref{geo:hitsep} plays the role of~\cite[Lemma~3.3]{DK}, while the extreme-neighbour conclusion in Lemma~\ref{geo:cluster} corresponds to Lemma~A.1. The first variation formula~\eqref{geo:firstvariation} and Lemma~\ref{prob:second-variation} play the role of Lemma~A.2. Proposition~\ref{geo:detector}, followed by the Fourier approximation in Lemma~\ref{prob:jackson} and the extraction of a coordinate in~\eqref{prob:coordinate}, corresponds to Proposition~A.3. Lemma~\ref{prob:small-values} is its Lemma~A.4. Proposition~\ref{prob:gaps} supplies the probability-one conclusion obtained there in Appendix~A.4. We now complete the proof by the translation argument of Appendix~A.5, keeping track of arbitrary tensor translations.

First, the asserted set is Borel. Indeed, the set of unoriented free homotopy classes is countable, and the length in each class is continuous in the metric. For $\beta=B\sqrt{k}$, the set is
\[
 \bigcup_{q=1}^\infty\ \bigcap_{\gamma\ne\eta}\left\{g\in\cU\cap C^k:|\ell_g(\gamma)-\ell_g(\eta)|\ge q^{-1}e^{-\beta\max\{\ell_g(\gamma),\ell_g(\eta)\}}\right\}.
\]
Each condition in the intersection is closed relative to $\cU\cap C^k$. Denote the complement of this Borel set in $\cU\cap C^k$ by $E$.

We next prove that $E$ is shy locally. Choose a $C^k$-ball $B(g_0,r)$ with $B(g_0,3r)\subset\cU$, and take the probe~\eqref{prob:probe} with compact support $K$ contained in the radius-$r$ ball about zero. For any tensor $a$, suppose that $(a+K)\cap B(g_0,r)$ is nonempty, and choose $v_0\in K$ with $a+v_0\in B(g_0,r)$. Every $v\in K$ then satisfies
\[
 \norm{a+v-g_0}_{C^k}\le\norm{a+v_0-g_0}_{C^k}+\norm{v-v_0}_{C^k}<3r.
\]
The whole translated family lies in the bounded ball $B(g_0,3r)$. Proposition~\ref{prob:gaps} therefore gives $\mathbb P\{a+v_\xi\in E\cap B(g_0,r)\}=0$. If the translated support does not meet the smaller ball, the same equality is immediate. Thus the law of $v_\xi$ witnesses shyness of $E\cap B(g_0,r)$ for every tensor translation $a$. The translated metrics need not remain in one conformal class.

Since $C^k(S^2T^*M)$ is separable, countably many such smaller balls cover $\cU\cap C^k$. Let $E_j$ be the corresponding exceptional pieces. We finish by constructing one witness for their union. For each $j$, choose a random tensor $w_j$ witnessing shyness of $E_j$, with compact support of radius at most $2^{-j}$. This is possible because the probe size $\eps$ is arbitrary. Choose these tensors independently. The series $w=\sum_jw_j$ converges uniformly over the product of their supports, since its tail norm is at most $\sum_{j>L}2^{-j}$. Its sum is consequently a continuous image of a compact product, and its law has compact support.

For a fixed translation $a$ and an index $j$, conditioning on all $w_i$ with $i\ne j$ leaves one fixed translation of $w_j$. Its witnessing property gives
\[
 \mathbb P\left\{a+\sum_iw_i\in E_j\,\middle|\,(w_i)_{i\ne j}\right\}=0.
\]
Integration over the conditioned variables, followed by the countable union over $j$, shows that $\mathbb P\{a+w\in E\}=0$ for every $a$. This proves prevalence in $\cU\cap C^k$.

Finally, every nonempty $C^k$-ball contained in $\cU$ contains a translate of an arbitrarily small probe. Proposition~\ref{prob:gaps} provides a metric satisfying~\eqref{prob:gap-property} in that ball. The asserted set is therefore dense.
\end{proof}

\end{document}